\documentclass[12pt, twoside]{article}
\usepackage{amsmath,amsthm,amssymb}
\usepackage[utf8]{inputenc}
\usepackage{times}
\usepackage{enumerate}
\usepackage{needspace}
\usepackage{tikz}
\usetikzlibrary{calc,positioning,shadows.blur,decorations.pathreplacing,patterns}
\usepackage{wrapfig}

\newcommand{\url}[1]{\texttt{#1}}

 \newcommand*{\double}[2][.1ex]{%
  \mathrel{\vcenter{\offinterlineskip%
  \hbox{$#2$}\vskip#1\hbox{$#2$}}}}
\newcommand{\txtand}{{\mathop{\text{~and~}}}}

\newcommand{\Bcal}{{\mathcal{B}}}

\newcommand{\Fcal}{{\mathcal{F}}}
\newcommand{\Gcal}{{\mathcal{G}}}
\newcommand{\Hcal}{{\mathcal{H}}}
\newcommand{\Ical}{{\mathcal{I}}}

\newcommand{\Mcal}{{\mathcal{M}}}

\newcommand{\Xcal}{{\mathcal{X}}}

\newcommand{\Fbb}{{\mathbb{F}}}

\newcommand{\Pbf}{{\mathbf{P}}}

\newcommand{\Tbf}{{\mathbf{T}}}

\newcommand{\restrict}{|}
\newcommand{\contract}{.}

\newcommand{\BS}{\backslash}

\newcommand{\rk}{\mathrm{rk}}
\newcommand{\cl}{\mathrm{cl}}

\newcommand{\Z}{\mathbb{Z}}

\renewcommand{\phi}{\varphi}

\newcommand{\maparrow}{\longrightarrow}

\newcommand{\SET}[1]{{\left\{ #1 \right\}}}

\newcommand{\dSET}[1]{{\left\{ #1 \right\}_{\neq}}}

\newcommand{\COMMENT}[1]{}
\newcommand{\XXXCUTXXX}[1]{}

\newcommand{\ROMANENUM}{\renewcommand{\theenumi}{(\roman{enumi})}\renewcommand{\labelenumi}{\theenumi}}

\newcommand{\routesto}{\double{\rightarrow}}

\newcommand{\deftext}[2][]{\emph{#2}}
\newcommand{\deftextX}[2][]{\emph{#2}}

\newcommand{\PRFR}[1]{\ignorespaces}
\newcommand{\bm}{\ignorespaces}

\def\titlerunning#1{\gdef\titrun{#1}}
\makeatletter
\def\author#1{\gdef\autrun{\def\and{\unskip, }#1}\gdef\@author{#1}}
\def\address#1{{\def\and{\\\hspace*{18pt}}\renewcommand{\thefootnote}{}%
\footnote {#1}}%
\markboth{\autrun}{\titrun}}
\makeatother
\def\email#1{e-mail: #1}
\def\subjclass#1{{\renewcommand{\thefootnote}{}%
\footnote{\emph{Mathematics Subject Classification (2010):} #1}}}
\def\keywords#1{\par\medskip
\noindent\textbf{Keywords.} #1}

\newtheorem{theorem}{Theorem}[section]
\newtheorem{corollary}[theorem]{Corollary}

\newtheorem{lemma}[theorem]{Lemma}

\theoremstyle{definition}
\newtheorem{definition}[theorem]{Definition}
\newtheorem{remark}[theorem]{Remark}

\newtheorem{step}{Step}

\makeatletter
\renewenvironment{cases}[1][l]{\matrix@check\cases\env@cases{#1}}{\endarray\right.}
\def\env@cases#1{%
  \let\@ifnextchar\new@ifnextchar
  \left\lbrace\def\arraystretch{1.2}%
  \array{@{}#1@{\quad}l@{}}}

\makeatother

\begin{document}

% TITLE PAGE

\titlerunning{Deciding Gammoid Class-Membership}

\title{Well-Scaling Procedure for Deciding Gammoid Class-Membership of Matroids}

\author{Immanuel Albrecht}

\date{\today}

\maketitle

\address{I.~Albrecht: %Hochschulstr.~36, D-01069~Dresden, Germany; \email{Immanuel.Albrecht@fernuni-hagen.de}
FernUniversität in Hagen, Fakultät für Mathematik und Informatik, Lehrgebiet für Diskrete Mathematik und Optimierung, D-58084~Hagen, Germany;
\email{Immanuel.Albrecht@fernuni-hagen.de}
}


% ABSTRACT

\begin{abstract}
We introduce a procedure that solves the decision problem whether a given matroid $M$ is a gammoid.
The procedure consists of three pieces: First, we introduce a notion of a valid matroid tableau which captures the current
state of knowledge regarding the properties of matroids related to the matroid under consideration. 
Second, we give a sufficient set of rules
that may be used to generate valid matroid tableaux. Third, we introduce a succession of steps that ultimately lead to a 
decisive tableau starting with any valid tableau. We argue that the decision problem scales well with respect to parallel computation models.

%% Keywords are optional
\end{abstract}

The \emph{Gammoid Class-Membership Problem} is the following decision problem: Given a matroid $M=(E,\Ical)$,
determine whether $M$ is a gammoid or not.
It is a well-known fact that the class of gammoids is closed under duality, minors, and direct sums; and that it may not be
characterized by a finite number of excluded minors.
D.~Mayhew even showed that every gammoid is a minor of some excluded minor of the class of gammoids \cite{Ma16},
therefore any attempt to solve this problem relying solely on excluded minors appears to be futile.
We introduce a decision procedure for Gammoid Class-Membership Problems that is guaranteed to ultimately give an answer through
exhaustive search, and which is also capable to incorporate knowledge of non-gammoids and strict gammoids in order to give an answer
before exhausting the search space in many cases.
Furthermore, the derivation steps described in our process may be carried out using massive parallelism,
since joining tableaux is a valid derivation.

% CONTENT

\section{Preliminaries}

In this work, we consider \emph{matroids} to be pairs $M=(E,\Ical)$ where $E$ is a finite set 
and $\Ical$ is a system of
independent subsets of $E$ subject to the usual axioms (\cite{Ox11}, Sec.~1.1).
The family of bases of $M$ shall be denoted by $\Bcal(M)$,
the family of flats of $M$ shall be denoted by $\Fcal(M)$.
If $M=(E,\Ical)$ is a matroid and $X\subseteq E$, then the restriction of $M$ to $X$
shall be denoted by $M\restrict X$ (\cite{Ox11}, Sec.~1.3).
%, and the contraction of $M$ to $X$ shall be denoted by $M\contract X$
%(\cite{Ox11}, Sec.~3.1).
A matroid $N=(E',\Ical')$ is an \deftext{extension} of $M$, if $E\subseteq E'$ 
and $\Ical = \SET{X\in\Ical'~\middle|~ X\subseteq E}$ holds.
The \emph{dual matroid} of $M$ shall be denoted by $M^\ast$.
A \deftext{modular cut} of $M$ is a set $C\subseteq \Fcal(M)$ that is closed under super-flats and
under the intersection of pairs of modular flats.
H.H.~Crapo showed, that there is a one-to-one correspondence between single-element extensions of a matroid $M$
and its modular cuts \cite{C65}.

Furthermore, the notion of a \emph{digraph} shall be synonymous with what is described more 
precisely as \emph{finite simple directed graph} that may have some loops, i.e. a digraph is 
a pair $D=(V,A)$ where $V$ is a finite
set and $A\subseteq V\times V$. 
%Every digraph $D=(V,A)$ has a unique \emph{opposite digraph} $D^\opp = (V,A^\opp)$ where
%$(u,v)\in A^\opp$ if and only if $(v,u)\in A$.
All standard notions related to digraphs in this work are in
accordance with the definitions found in \cite{BJG09}. A \emph{path} in $D=(V,A)$ is
a non-empty and non-repeating sequence $p = p_1 p_2 \ldots p_n$ of vertices $p_i\in V$ such that
for each $1 \leq i < n$, $(p_i,p_{i+1})\in A$. By convention, we shall denote $p_n$ by $p_{-1}$.
Furthermore, the set of vertices traversed by a path $p$ shall be denoted by $\left| p \right| = \SET{p_1,p_2,\ldots,p_n}$
and the set of all paths in $D$ shall be denoted by $\Pbf(D)$. For $D=(V,A)$ and $S,T\subseteq V$, an \emph{$S$-$T$-separator}
is a set $X\subseteq V$ such that every path $p\in\Pbf(D)$ from $s\in S$ to $t\in T$ has $\left| p \right| \cap V\not=\emptyset$.

\begin{definition}\PRFR{Jan 22nd}
  Let $D = (V,A)$ be a digraph, and $X,Y\subseteq V$. A \deftext{routing} from $X$ to $Y$ in $D$ is a family of paths $R\subseteq \Pbf(D)$ such that
  \begin{enumerate}\ROMANENUM
    \item for each $x\in X$ there is some $p\in R$ with $p_{1}=x$,
    \item for all $p\in R$ the end vertex $p_{-1}\in Y$, and
    \item for all $p,q\in R$, either $p=q$ or $\left|p\right|\cap \left|q\right| = \emptyset$.%, and
        %\item all $p\in R$ are simple.
  \end{enumerate}
  We shall write $R\colon X\routesto Y$ in $D$ as a shorthand for ``$R$ is a routing from $X$ to $Y$
    in $D$'', and if no confusion is possible, \label{n:routing}
    we just write $X\routesto Y$ instead of $R$ and $R\colon X\routesto Y$.
%    A routing $R$ is called \deftext{linking} from $X$ to $Y$, if it is a routing onto $Y$, i.e. whenever $Y = \SET{p_{-1}\mid p\in R}$.
\end{definition}

\begin{definition}\label{def:gammoid}\PRFR{Jan 22nd}
    Let $D = (V,A)$ be a digraph, $E\subseteq V$,
    and $T\subseteq V$. 
    The \deftext[gammoid represented by DTE@gammoid represented by $(D,T,E)$]{gammoid represented by $\bm{(D,T,E)}$} is defined to be the matroid $\Gamma(D,T,E)=(E,\Ical)$\label{n:GTDE}
     where
    \[ \Ical = \SET{X\subseteq E \mid \text{there is a routing } X\routesto T \text{ in D}}. \]
    The elements of $T$ are usually called \deftextX{sinks} in this context, although they are not required to be actual sinks of the digraph $D$. To avoid confusion, 
    we shall call the elements of $T$ \deftext{targets} in this work. A matroid $M'=(E',\Ical')$ is called \deftextX{gammoid}, if there is a digraph $D'=(V',A')$ and a set $T'\subseteq V'$ such that $M' = \Gamma(D',T',E')$.
    A gammoid $M$ is called \deftext{strict}, if there is a representation $(D,T,E)$ of $M$ with $D=(V,A)$ where $V=E$.
\end{definition}

\begin{definition}\PRFR{Mar 29th}
  Let $M=(E,\Ical)$ be a matroid. Then $M$ shall be \deftext{strongly base-orderable}, if for every pair of bases
  $B_1,B_2\in\Bcal(M)$ there is a bijective map $\phi\colon B_1\maparrow B_2$ such that
  \( \left( B_1\BS X  \right) \cup \phi[X] \in \Bcal(M) \)
  holds
  for all $X\subseteq B_1$.
\end{definition}

\begin{lemma}[\cite{M72}, Corollary~4.1.4]\label{lem:GammoidsAreStronglyBaseOrderable}\PRFR{Mar 29th}
  Let $M=(E,\Ical)$ be a gammoid. Then $M$ is strongly base-orderable.
\end{lemma}
For a proof, see \cite{M72}.

\begin{definition}\label{def:alphaM}\PRFR{Jan 30th}
  Let $M=(E,\Ical)$ be a matroid. The \deftext[a-invariant of M@$\alpha$-invariant of $M$]{$\bm \alpha$-invariant of $\bm M$} shall be the map\label{n:alphaM}
  $ \alpha_M\colon 2^E \maparrow \Z $
  that is uniquely characterized by the recurrence relation
  \[ \alpha_M(X) = \left| X \right| - \rk_M(X) - \sum_{F\in \Fcal(M,X)} \alpha_M(F),\]
  where $\Fcal(M,X) = \SET{F\in\Fcal(M)~\middle|~ F\subsetneq X}$.
\end{definition}

\begin{theorem}[\cite{M72}, Theorems~2.2 and 2.4]\label{thm:strictGammoids}
  Let $M=(E,\Ical)$ be a matroid. Then $M$ is a strict gammoid if and only if $\alpha_M \geq 0$.
\end{theorem}
For a proof, see \cite{M72}.

\begin{theorem}[\cite{In77}, Theorem~13; \cite{Brylawski1971}, \cite{Brylawski1975}, \cite{Ingleton1971}]\label{thm:graphicGammoidsAreSeriesParallel}\PRFR{Mar 29th}
  Let $\Fbb_2$ be the two-elementary field, $E,C$ finite sets, and let $\mu \in \Fbb_2^{E\times C}$ be a matrix.
  Then the linear matroid $M(\mu)$ is a gammoid if and only if there is no minor $N$ of $M(\mu)$ which is isomorphic to $M(K_4)$.
  The latter is the case if and only if $M(\mu)$ is isomorphic to the polygon matroid of a series-parallel network.
\end{theorem}

 For proofs of a sufficient set of implications which establish the equivalency stated, refer to
 \cite{Brylawski1971}, \cite{Brylawski1975}, and \cite{Ingleton1971}.

 \begin{theorem}[\cite{Ox11}, Theorem~6.5.4]\label{thm:binaryMatroids}\PRFR{Mar 29th}
  Let $M=(E,\Ical)$ be a matroid. Then $M$ is isomorphic to the linear matroid $M(\mu)$ for some matrix $\mu \in \Fbb_2^{E\times C}$
  if and only if $M$ has no minor isomorphic to the uniform matroid $U_{2,4} = \left( E',\Ical' \right)$,
  where $E'= \dSET{a,b,c,d}$ and $\Ical' = \SET{\vphantom{A^A}X\subseteq E'~\middle|~ \left| X \right| \leq 2}$.
 \end{theorem}
  See \cite{Ox11}, pp.193f, for a proof.

\begin{definition}\label{def:deflateOfM}\PRFR{Mar 27th}
  Let $M=(E,\Ical)$ be a matroid, $X\subseteq E$. The restriction $N=M\restrict X$ shall be a \deftext[deflate of a matroid]{deflate of $\bm M$},
  if $E\BS X = \dSET{e_1,e_2,\ldots,e_m}$ can be ordered naturally,
  such that for all $i\in \SET{1,2,\ldots,m}$ the modular cut
  \[ C_i = \SET{F\in \Fcal\left( \vphantom{A^A} M\restrict \left( X\cup\SET{e_1,e_2,\ldots,e_{i-1}} \right)  \right)~\middle|~ e_i \in \cl_M(F)} \]
  has precisely one $\subseteq$-minimal element.
  $M$ shall be called \deftext{deflated}, if the only deflate of $M$ is $M$ itself.
\end{definition}

\begin{lemma}\label{lem:deflationLemma}\PRFR{Mar 27th}
  Let $M=(E,\Ical)$ be a matroid, $X\subseteq E$ and let  $N = M\restrict X$ be a deflate of $M$.
  Then $M$ is a gammoid if and only if $N$ is a gammoid.
\end{lemma}
\begin{proof}\PRFR{Mar 27th}
  If $M$ is a gammoid, then $N$ is a gammoid, since the class of gammoids is closed under minors (\cite{M72}, Sec.~1 and Cor.~4.1.3).
  Now let $N$ be a gammoid, and let $E\BS X = \dSET{e_1,e_2,\ldots,e_m}$ be implicitly ordered with the properties required in Definition~\ref{def:deflateOfM}.
  We proof the statement of this lemma by induction on $\left| E\BS X \right| = m$.
  The base case $m = 0$ is trivial, the induction step follows from the special case where $E\BS X = \SET{e_1}$.
  Let $F_1 = \bigcap C_1$ be the unique minimal element of the modular cut $C_1$. Then $M$ arises from $N$ by adding a new point $e_1$
  to $N$, which is in general position with respect to the flat $F_1$. Let $(D,T,X)$ be a representation of $N$ with $D=(V,A)$ and $e_1\notin V$.
  Let $D' = \left( V\cup \SET{e_1}, A\cup \left( \SET{e_1}\times F_1 \right)  \right)$. It is easy to see that
  $(D',T,E)$ is a representation of $M$.
\end{proof}

\begin{theorem}[\cite{KW12}, Theorem~3]\label{thm:upperBoundSizeOfV}\PRFR{Mar 7th}
  Let $D=(V,A)$ be a digraph, $E,T\subseteq V$, and $r>0$ be the cardinality of a minimal $E$-$T$-separator in $D$.
  There is a set $Z\subseteq V$ with $E\cup T \subseteq Z$ and
   $\left| Z \right| = O(\left| E \right|\cdot\left| T \right|\cdot r)$
  such that for all $X\subseteq E$ and $Y\subseteq T$ there is a minimal $X$-$Y$-separator $S$ in $D$ with $S\subseteq Z$.
\end{theorem}

\PRFR{Mar 7th}
\noindent
For the proof, see \cite{KW12}, where the authors only give the $O$-behavior of the size of $Z$ in Theorem~\ref{thm:upperBoundSizeOfV},
   but it is possible to derive the factor hidden in the $O$-notation by inspecting their proof and the proof of
   \cite{Marx09} Lemma~4.1.
  We obtain that $E\cup T \subseteq Z$ and
   \[ \left| Z \right| \leq \binom{r}{1} \cdot \binom{\left| E \right|}{1}  \cdot \binom{\left| T \right|}{1} + \left| E \right| + \left| T \right| = r\cdot \left| E \right|\cdot \left| T \right| + \left| E \right| + \left| T \right|. \]
  Let $(D,T,E)$ be a representation of $M$ where $\left| T \right| = \rk_M(E)$ and $D=(V,A)$.
  Let $Z\subseteq V$ be a subset of $V$ as in the consequent 
  of Theorem~\ref{thm:upperBoundSizeOfV}.
  Let $D'=(Z,A')$ be the digraph,
  where for all $x,y\in Z$, there is an arc
  \[ (x,y)\in A' \quad\Longleftrightarrow\quad \exists p\in \Pbf(D; x,y) \colon\, \left| p \right|\cap Z = \SET{x,y}.\]
  Thus there is an arc leaving $y\in Z$ and entering $z\in Z$ in $D'$ if there is a path from $y$ to $z$ in $D$ that never visits
  another vertex of $Z$. It is routine to show that $(D',T,E)$ represents the same matroid as $(D,T,E)$. Therefore we obtain:
  \begin{corollary}\label{cor:NbrOfV}
  Let $M=(E,\Ical)$ be a gammoid.
  There is a representation $(D,T,E)$ of $M$ where $D=(V,A)$ such that $\left| T \right| = \rk_M(E)$
  and such that $ \left| V \right| \leq \rk_M(E)^2 \cdot \left| E \right| + \rk_M(E) + \left| E \right| $.
  \end{corollary}

\section{Matroid Tableaux}
\begin{definition}\PRFR{Mar 29th}
  A \deftext{matroid tableau} is a tuple \label{n:mattab} $\Tbf = (G,\Gcal,\Mcal,\Xcal,\simeq)$ where
  \begin{enumerate}\ROMANENUM
    \item $G$ is a matroid, called the \deftextX{goal} of $\bm \Tbf$,
    \item $\Gcal$ is a family of matroids, called the \deftextX{gammoids} of $\bm \Tbf$,
    \item $\Mcal$ is a family of matroids, called the \deftextX{intermediates} of $\bm \Tbf$,
    \item $\Xcal$ is a family of matroids, called the \deftextX{excluded matroids} of $\bm \Tbf$, and where
    \item $\simeq$ is an equivalence relation on $\SET{G'~\middle|~ G'\text{~is a minor of~}G}\cup\Gcal \cup \Mcal \cup \Xcal$, called the \deftextX{equivalence} of $\bm \Tbf$. \qedhere
  \end{enumerate}
\end{definition}

\needspace{5\baselineskip}
\begin{definition}\label{def:validTableau}\PRFR{Mar 29th}
  Let $\Tbf = (G,\Gcal,\Mcal,\Xcal,\simeq)$ be a matroid tableau.
  $\Tbf$ shall be \deftext[valid matroid tableau]{valid},
  \begin{enumerate}\ROMANENUM
  \item  if
  all matroids in $\Gcal$ are indeed gammoids,
  \item if no matroid in $\Mcal$ is a strict gammoid,
  \item if all matroids in $\Xcal$ are indeed matroids which are not gammoids, and
  \item  if for every equivalence classes $[M]_\simeq$ of $\simeq$ we have that either $[M]_\simeq$ is fully contained in the class of gammoids
  or $[M]_\simeq$ does not contain a gammoid. \qedhere
\end{enumerate}
\end{definition}

\begin{definition}\label{def:decisiveTableau}\PRFR{Mar 29th}
  Let $\Tbf = (G,\Gcal,\Mcal,\Xcal,\simeq)$ be a matroid tableau. $\Tbf$ shall be \deftext[decisive matroid tableau]{decisive},
  if $\Tbf$ is valid and 
  if either of the following holds:
  \begin{enumerate}\ROMANENUM 
  \item There is a matroid $M\in \Gcal$ such that $G \simeq M$.
  \item There is 
  a matroid $X\in \Xcal$ that is isomorphic to a minor of $G$.
  \item For every extension $N=(E',\Ical')$ of $G=(E,\Ical)$ 
  with $$\left| E' \right| = \rk_G(E)^2\cdot \left| E \right| + \rk_G(E) + \left| E \right|$$ there is a matroid $M\in\Mcal$ that is
  isomorphic to $N$. \qedhere
\end{enumerate}
\end{definition}

\needspace{3\baselineskip}
\begin{lemma}\label{lem:decisiveTableau}\PRFR{Mar 29th}
  Let $\Tbf = (G,\Gcal,\Mcal,\Xcal,\simeq)$ be a decisive matroid tableau. Then $G$ is a gammoid if and only if there is a matroid $M\in \Gcal$
  such that $G\simeq M$.
\end{lemma}
\begin{proof}\PRFR{Mar 29th}
  Assume that such an $M\in \Gcal$ exists. From Definition~\ref{def:validTableau}
  we obtain that $M$ is a gammoid, and that in this case $G\simeq M$ implies that $G$ is a gammoid, too.
  Now assume that no $M\in \Gcal$ has the property $G\simeq M$. Since $\Tbf$ is decisive, either case {\em (ii)} or {\em (iii)} of
  Definition~\ref{def:decisiveTableau} holds. If case {\em (ii)} holds, then $G$ cannot be a gammoid since it has a non-gammoid minor,
  but the class of gammoids is closed under minors.
  If case {\em (iii)} holds but not case {\em (ii)}, then no extension of $G=(E,\Ical)$ with $k=\rk_G(E)^2\cdot \left| E \right| $ $+\, \rk_G(E) + \left| E \right|$
  elements is a strict gammoid. Now assume that $G$ is a gammoid, then there is a digraph $D=(V,A)$ with $\left| V \right| \leq k$ vertices, such that
  $G = \Gamma(D,T,E)$
  for some $T\subseteq V$ (Corollary~\ref{cor:NbrOfV}). 
  Let $N' = \Gamma(D,T,V)\oplus (V',\SET{\emptyset})$ with $V'\cap V = \emptyset$ and $\left| V' \right| + \left| V \right| = k$. 
  Clearly, $N'$ is an extension of $G$ on a ground set with $k$ elements, which is also a strict gammoid, a contradiction to the assumption that $N'$ is 
  isomorphic to some $N\in \Mcal$, since $\Mcal$ is a family which consists of matroids that are not strict gammoids. Therefore we may conclude that in case {\em (iii)} the matroid $G$ is not a gammoid.
\end{proof}

\section{Valid Derivations}

\PRFR{Mar 29th}
 A \deftext{derivation} is an operation on a finite number of input tableaux and possible additional parameters with constraints
that produces an output tableau. Furthermore,
a derivation is \deftext[valid derivation]{valid}, if the output tableau is valid for all sets of valid input tableaux and possible additional parameters that
satisfy the constraints.
% The valid derivations presented here are fairly straight-forward consequences of the concepts presented earlier in this work.

\begin{definition}\PRFR{Mar 29th}
  Let $\Tbf_i = (G_i,\Gcal_i,\Mcal_i,\Xcal_i,\simeq^{(i)})$ be matroid tableaux for $i\in \SET{1,2,\ldots,n}$.
  The \deftext{joint tableau} shall be the matroid tableaux \label{n:jointTableau}
  \[\bigcup_{i=1}^{n} \Tbf_i = (G_1,\Gcal,\Mcal,\Xcal,\simeq)\]
  where \[ \Gcal = \bigcup_{i=1}^n \Gcal_i,\,\,\, \Mcal = \bigcup_{i=1}^n \Mcal_i,\,\,\, \Xcal = \bigcup_{i=1}^n \Xcal_i, \]
  and where $\simeq$ is the smallest equivalence relation such that $M \simeq^{(i)} N$ implies $M \simeq N$ for all $i\in \SET{1,2,\ldots,n}$.
  In other words, $\simeq$ is the equivalence relation on the family of matroids
   $\SET{G'~\middle|~ G'\text{~is a minor of~}G}\cup\Gcal \cup \Mcal \cup \Xcal$ which is
  generated by the relations $ \simeq^{(1)}, \simeq^{(2)}, \ldots, \simeq^{(n)}$.
\end{definition}

\needspace{2\baselineskip}
\begin{lemma}\PRFR{Mar 29th}
  The derivation of the joint tableau is valid.
\end{lemma}
\begin{proof}\PRFR{Mar 29th}
  Clearly, $\Gcal$, $\Mcal$, and $\Xcal$ inherit their desired properties of Definition~\ref{def:validTableau} from
  the valid input tableaux $\Tbf_i$  where $i\in\SET{1,2,\ldots, n}$. Now let $M \simeq N$ with $M\not= N$.
  Then there are matroids $M_1,M_2,\ldots,M_{k}$ and indexes $i_0,i_1,\ldots,i_k\in \SET{1,2,\ldots,n}$ such that
  there is a chain of $\simeq^{(i)}$-relations
  \[ M \simeq^{(i_0)} M_1 \simeq^{(i_1)} M_2 \simeq^{(i_2)} \cdots \simeq^{(i_{k-1})} M_k \simeq^{(i_k)} N. \]
  The assumption that the input tableaux are valid yields that $M$ is a gammoid if and only if $M_1$ is a gammoid,
  if and only if $M_2$ is a gammoid, and so on. Therefore it follows that $M$ is a gammoid if and only if $N$ is a gammoid,
  thus $\simeq$ has the desired property of Definition~\ref{def:validTableau}. Consequently, $\bigcup_{i=1}^n \Tbf_i$ is a valid tableau.
\end{proof}

\begin{definition}\PRFR{Mar 29th}
  Let $\Tbf = (G,\Gcal,\Mcal,\Xcal,\simeq)$ and $\Tbf' = (G,\Gcal',\Mcal',\Xcal',\simeq')$ be matroid tableaux.
  We say that $\Tbf$ is a \deftext[sub-tableau]{sub-tableau of $\Tbf\bm'$} if $\Gcal \subseteq \Gcal'$, $\Mcal \subseteq \Mcal'$, and
  $\Xcal \subseteq \Xcal'$ holds, and if $M \simeq N$ implies $M \simeq' N$.
\end{definition}

\needspace{2\baselineskip}
\begin{lemma}\PRFR{Mar 29th}
  The derivation of a sub-tableau is valid.
\end{lemma}
\begin{proof}\PRFR{Mar 29th}
  Clearly $\Tbf$ inherits the properties of Definition~\ref{def:validTableau} from the validity of $\Tbf'$.
\end{proof}

\begin{definition}\PRFR{Mar 29th}
  Let $\Tbf = (G,\Gcal,\Mcal,\Xcal,\simeq)$  be a matroid tableau. We shall call the
  \label{n:expTab}
   tableau $[\Tbf]_\simeq= (G,\Gcal',\Mcal,\Xcal',\simeq)$   \deftext[expansion tableau]{expansion tableau of $\Tbf$} 
  whenever \[ \Gcal' = \bigcup_{M\in\Gcal} [M]_\simeq \quad\txtand\quad \Xcal' = \bigcup_{M\in\Xcal} [M]_\simeq. \qedhere\]
\end{definition}

\needspace{2\baselineskip}
\begin{lemma}\PRFR{Mar 29th}
  The derivation of the expansion tableau is valid.
\end{lemma}
\begin{proof}\PRFR{Mar 29th}
  If $M'\in \Gcal'$, then there is some $M\in \Gcal$ such that $M\simeq M'$. Since we assume $\Tbf$ to be valid, we may infer that
  $M'$ is a gammoid if and only if $M$ is a gammoid, and the latter is the case since $M\in\Gcal$. Therefore $M'$ is a gammoid.
  An analogous argument yields that if $M'\in \Xcal'$, then $M'$ is not a gammoid.
\end{proof}

\needspace{5\baselineskip}
\begin{definition}\PRFR{Mar 29th}
  Let $\Tbf = (G,\Gcal,\Mcal,\Xcal,\simeq)$  be a matroid tableau. We shall call the
  \label{n:extTab}
   tableau $[\Tbf]_{\equiv} = (G,\Gcal',\Mcal',\Xcal',\simeq')$ \deftext[extended tableau]{extended tableau of $\bm \Tbf$}
   whenever $$\Gcal' = \Gcal \cup \SET{M^\ast~\middle|~M\in \Gcal},\,\,\,
        \Xcal' = \Xcal \cup \SET{M^\ast~\middle|~M\in \Xcal},\,\,\, 
     \Mcal' = \Mcal \cup \Xcal',$$ and when 
     $\simeq'$ is the smallest equivalence relation that contains the 
     relations $\simeq$ and $\sim$; where $M\sim N$ if and only if $N$ is
     isomorphic to $M$ or $M^\ast$.
\end{definition}

\needspace{2\baselineskip}
\begin{lemma}\PRFR{Mar 29th}
  The derivation of the extended tableau is valid.
\end{lemma}
\begin{proof}\PRFR{Mar 29th}
  The class of gammoids is closed under duality, therefore a matroid $M$ is a gammoid if and only if $M^\ast$ is a gammoid. So $\Gcal'$ and $\Xcal'$ inherit their desired properties of Definition~\ref{def:validTableau} from the validity of $\Tbf$.
  If $M\in \Mcal'\BS \Mcal$, then $M\in \Xcal'$, therefore $M$ cannot be a strict gammoid.
\end{proof}

\begin{definition}\label{def:decisionTableau}\PRFR{Mar 29th}
  Let $\Tbf = (G,\Gcal,\Mcal,\Xcal,\simeq)$ be a decisive matroid tableau. The tableau
  \label{n:concTab}
   $\Tbf! = (G,\Gcal',\Mcal,\Xcal',\simeq)$ shall be the \deftext[conclusion tableau]{conclusion tableau for $\Tbf$} if either
  \begin{enumerate}\ROMANENUM
  \item $\Gcal' = \Gcal \cup \SET{G'~\middle|~ G'\text{~is a minor of~} G}$, $\Xcal' = \Xcal$, and the tableau $\Tbf$ 
  satisfies case {(i)} of Definition~\ref{def:decisiveTableau}; or
  \item $\Gcal' = \Gcal$, $\Xcal' = \Xcal\cup\SET{G}$, and $\Tbf$ satisfies case {(ii)} or {(iii)} of Definition~\ref{def:decisiveTableau}. \qedhere
  \end{enumerate}
\end{definition}

\needspace{2\baselineskip}
\begin{corollary}\PRFR{Mar 29th}
  The derivation of the conclusion tableau is valid.
\end{corollary}
\begin{proof}\PRFR{Mar 29th}
  Easy consequence of Lemma~\ref{lem:decisiveTableau}.
\end{proof}

\begin{definition}\PRFR{Mar 29th}
  Let $\Tbf = (G,\Gcal,\Mcal,\Xcal,\simeq)$  be a matroid tableau, let $M_1$ and $M_2$ be matroids of the
  tableau, i.e.
  \[ \SET{M_1,M_2}\subseteq  \SET{G'~\middle|~ G'\text{~is a minor of~}G}\cup\Gcal \cup \Mcal \cup \Xcal .\]
  Furthermore, let $M_1$ be a deflate of $M_2$.
  The tableau
  \label{n:idTab}
   $$\Tbf(M_1\simeq M_2) = (G,\Gcal,\Mcal,\Xcal,\simeq')$$ is called 
  \deftext[identified tableau]{identified tableau} for $\Tbf$ with respect to $\bm M_1$ and $\bm M_2$ if
  the relation
  $\simeq'$ is the smallest equivalence relation, such that $M_1\simeq' M_2$ holds, and such that  $M'\simeq N'$  implies $M'\simeq' N'$.
\end{definition}

\needspace{2\baselineskip}
\begin{lemma}\PRFR{Mar 29th}
  The derivation of an identified tableau is valid.
\end{lemma}
\begin{proof}\PRFR{Mar 29th}
  Follows from Lemma~\ref{lem:deflationLemma}.
\end{proof}

\subsection{Valid Tableaux}

\begin{corollary}\label{cor:strictGammoidTableau}\PRFR{Mar 29th}
  Let $M=(E,\Ical)$ be a matroid with $\alpha_M \geq 0$. Then the matroid tableau
  $\Tbf$ is valid, where
  $\Tbf = (M,\Gcal,\Mcal,\Xcal,\simeq)$ with
  \( \Gcal = \SET{M,M^\ast}\),  $\Mcal= \emptyset$, $\Xcal=\emptyset$, and $M\simeq N \Leftrightarrow M=N$.
\end{corollary}
\begin{proof}\PRFR{Mar 29th}
See Theorem~\ref{thm:strictGammoids}.
\end{proof}

\begin{corollary}\PRFR{Mar 29th}
  Let $M=(E,\Ical)$ be a matroid with $\rk_M(X) = 3$, $X\subseteq E$ with $\alpha_M(X) < 0$. Then the matroid tableau
  $\Tbf$ is valid, where
  $\Tbf = (M,\Gcal,\Mcal,\Xcal,\simeq)$ with
  \( \Gcal = \emptyset\),  $\Mcal= \emptyset$, $\Xcal=\SET{M,M^\ast}$, and $M\simeq N \Leftrightarrow M=N$.
\end{corollary}
\begin{proof}\PRFR{Mar 29th}
Follows from Theorem~\ref{thm:strictGammoids} together with the fact that every gammoid of rank $3$ is
a strict gammoid (\cite{IP73}, Proposition~4.8).
\end{proof}

\begin{remark}\label{rem:nonStrictGammoidTableau}\PRFR{Mar 29th}
  Let $M=(E,\Ical)$ be a matroid, $X\subseteq E$ with $\alpha_M(X) < 0$. Then the matroid tableau
  $\Tbf$ is valid, where
  $\Tbf = (M,\Gcal,\Mcal,\Xcal,\simeq)$ with
  \( \Gcal = \emptyset\),  $\Mcal= \SET{M}$, $\Xcal=\emptyset$, and $M\simeq N \Leftrightarrow M=N$.
\end{remark}

\begin{corollary}\PRFR{Mar 29th}
  Let $M=(E,\Ical)$ be a matroid.
  If $M$ has no minor isomorphic to $M(K_4)$ and no minor isomorphic to $U_{2,4}$, then the matroid tableau
  $\Tbf$ is valid, where
  $\Tbf = (M,\Gcal,\Mcal,\Xcal,\simeq)$ with
  \( \Gcal = \SET{M,M^\ast}\),  $\Mcal= \emptyset$, $\Xcal=\emptyset$,  and $M\simeq N \Leftrightarrow M=N$.
\end{corollary}
\begin{proof}\PRFR{Mar 29th}
Direct consequence of Theorems~\ref{thm:graphicGammoidsAreSeriesParallel} and \ref{thm:binaryMatroids}.
\end{proof}

\begin{corollary}\PRFR{Mar 29th}
  Let $M=(E,\Ical)$ be a matroid, $B_1,B_2\in \Bcal(M)$ be bases of $M$ such that for every bijection
  $\phi\colon B_1\BS B_2 \maparrow B_2\BS B_1$ there is a set $X\subseteq B_1\BS B_2$
  with the property $ \left( B_1\BS X \right)\cup \phi[X] \notin \Bcal(M)$. Then the matroid tableau
  $\Tbf$ is valid, where
  $\Tbf = (M,\Gcal,\Mcal,\Xcal,\simeq)$ with
  \( \Gcal = \emptyset\),  $\Mcal= \emptyset$, $\Xcal=\SET{M,M^\ast}$, and $M\simeq N \Leftrightarrow M=N$.
\end{corollary}
\begin{proof}\PRFR{Mar 29th}
  Direct consequence of Lemma~\ref{lem:GammoidsAreStronglyBaseOrderable}.
\end{proof}

\subsection{Example}\label{ex:tab}

  Consider the matroid $G = G_{8,4,1} = (E,\Ical)$ where $E=\SET{1,2,\ldots,8}$ and where $\Ical = \SET{\vphantom{A^A}X\subseteq E ~\middle|~ \left| X \right|\leq 4,\,X\notin \Hcal}$ with
  \[\Hcal = \SET{\vphantom{A^A} \SET{1, 3, 7, 8},
  \SET{1, 5, 6, 8},
   \SET{2, 3, 6, 8},
  \SET{4, 5, 6, 7},
  \SET{2, 4, 7, 8}}.\]
  Clearly, $\alpha_G(H) = 1$ for all $H\in \Hcal$, and consequently $\alpha_G(E) = 4-5 = -1$.
  The dual matroid $G^\ast = (E,\Ical^\ast)$ has 
  $\Ical^\ast = \SET{\vphantom{A^A}X\subseteq E ~\middle|~ \left| X \right|\leq 4,\,X\notin \Hcal^\ast}$ with
  \[\Hcal^\ast = \SET{\vphantom{A^A} \SET{1, 2, 3, 8},
  \SET{1, 3, 5, 6},
   \SET{1, 4, 5, 7},
  \SET{2, 3, 4, 7},
  \SET{2, 4, 5, 6}}.\]
  Thus $\alpha_{G^\ast}(H') = 1$ for all $H'\in\Hcal^\ast$, and so $\alpha_{G^\ast}(E) = 4-5 = -1$, too.
   We start with the
  valid tableaux $$\Tbf_G = \left( G,\emptyset,\SET{G},\emptyset,\langle \, \rangle \right)
  \txtand \Tbf_{G^\ast} = \left( G^\ast,\emptyset,\SET{G^\ast},\emptyset,\langle \, \rangle \right),$$
  where $\langle .\rangle$ denotes the generated equivalence relation defined on the set of matroids occurring in the respective tableau.
  We may derive the extended joint tableau 
  $$ \Tbf_1 = [\Tbf_G \cup \Tbf_{G^\ast}]_\equiv = \left( G, \emptyset,\SET{G,G^\ast}, \emptyset, \langle G\simeq G^\ast\rangle  \right).$$
  Now observe that although $G$ is deflated, $G^\ast$ is not deflated. We have
     \begin{align*}C^\ast_8 & = \SET{\vphantom{A^A}F\in \Fcal\left( G^\ast\restrict \SET{1,2,\ldots,7}  \right)~\middle|~ 8 \in \cl_{G^\ast}(F)}
     \\ & = \SET{\vphantom{A^A}F\in \Fcal\left( G^\ast\restrict \SET{1,2,\ldots,7}  \right)~\middle|~\SET{1,2,3}\subseteq F}.
  \end{align*}
  Let $G^\ast_7 = G^\ast\restrict \SET{1,2,\ldots,7}$. We have $\alpha_{G^\ast_7}(\SET{1,2,\ldots,7}) = -1$, thus $G^\ast_7$ is not a 
  strict gammoid, and thus
  \[ \Tbf_{G^\ast_7} = \left(G^\ast_7, \emptyset, \SET{G^\ast_7}, \emptyset, \langle\,\rangle\right) \]
  is a valid tableau. Since $G^\ast_7$ is a deflate of $G^\ast$, each of them is an induced matroid with respect to the other. Therefore
  we may identify $G^\ast$ and $G^\ast_7$ in the joint tableau
  \[ \Tbf_2 = \left( \Tbf_1\cup \Tbf_{G^\ast_7} \right)(G^\ast\simeq G^\ast_7) = \left( G,\emptyset,\SET{G,G^\ast,G^\ast_7},\emptyset,\langle G \simeq G^\ast \simeq G^\ast_7\rangle \right).\]
  Now let $G_7 = \left( G^\ast_7 \right)^\ast$, and we have $\alpha_{G_7} \geq 0$. Thus
  \[ \Tbf_{G_7} = \left( G_7, \SET{G_7},\emptyset, \emptyset, \langle\,\rangle \right) \]
  is a valid tableau. We now may derive the decisive tableau
  \[ \Tbf_3 = [\Tbf_2 \cup \Tbf_{G_7}]_\equiv = \left( G, \SET{G_7}, \SET{G,G^\ast,G^\ast_7,G_7},\emptyset, \langle G\simeq G^\ast \simeq G^\ast_7 \simeq G_7\rangle \right) \]
  where case {\em (i)} of Definition~\ref{def:decisiveTableau} holds. Consequently, $G$ is a gammoid.

\begin{figure}
\includegraphics[width=\textwidth]{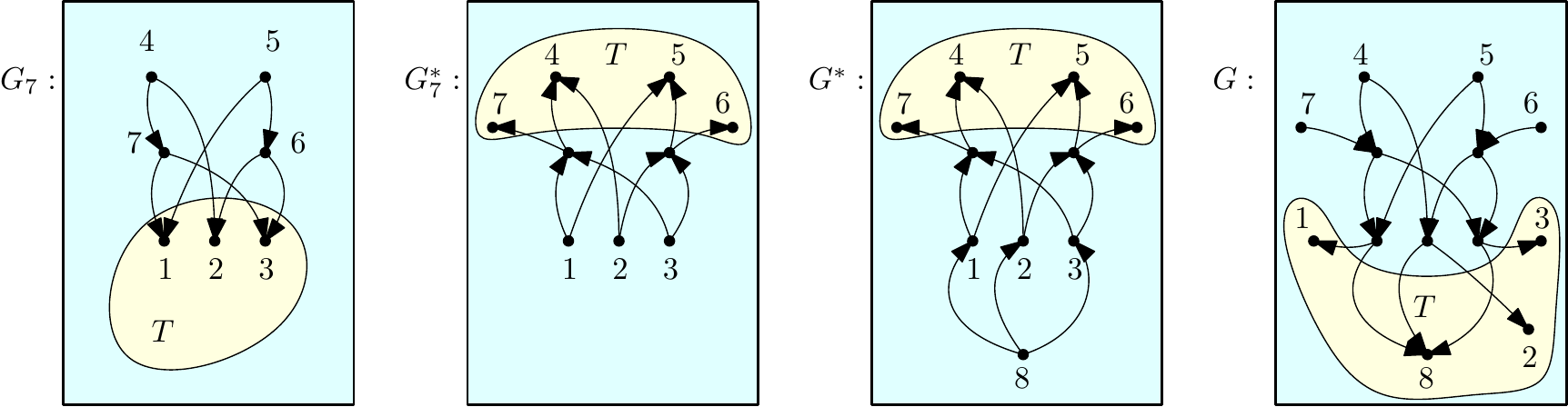}
\caption{\label{fig:tab}Reconstruction of a representation of $G_{8,4,1}$ from the matroid tableaux in Section~\ref{ex:tab}.}
\end{figure}

\section{Decision Procedure}

We may start the procedure with the valid initial tableau $\Tbf := (G,\emptyset,\emptyset,\emptyset,\langle\,\rangle)$,
or any other valid tableau that we may have obtained using some heuristic or intuitive derivation. For instance, if $\Tbf'$ is
a tableau obtained for a different goal $G'$, then the joint tableau $\Tbf := (G,\emptyset,\emptyset,\emptyset,\langle\,\rangle) \cup \Tbf'$
may be a better choice to start with, because $G$ and $G'$ may have common extensions and minors (up to isomorphy).

\begin{step}\label{step:start}\PRFR{Mar 31st}
  If $\Tbf$ is decisive, stop. If the procedure is run in a parallel fashion, you may choose to end
  a spawned thread here as long as there is another thread that carries on the computation.
\end{step}

\begin{step}\label{step:branchOut}\PRFR{Mar 31st}
  Choose an intermediate goal
  $M\in  \SET{G' ~\middle|~ G'\text{~is a minor of~}G} \cup \Mcal  $ such that $M\notin \Gcal \cup \Xcal$,
  preferably one with $M\simeq G$ which is small both in rank and cardinality.
  At this point, it is possible to spawn several parallel computations with multiple choices of $M$. 
  In this case, all subsequent updates of $\Tbf$ shall be considered atomic and synchronized.
\end{step}

\begin{step}\PRFR{Mar 31st}
  If $\Tbf_M = (M,\emptyset,\emptyset,\emptyset,\langle\,\rangle)\cup \Tbf$ is decisive,
  then set $\Tbf := \left[\vphantom{A^A}[\Tbf \cup \left( \Tbf_M!  \right)]_\equiv\right]_\simeq$
   and continue with Step~\ref{step:start}.
\end{step}

\begin{step}\label{step:choice}
  Determine whether $M$ has a minor that is isomorphic to $M(K_4)$. If this is the case,  then 
  $\Tbf_M = \left( M, \emptyset, \emptyset, \SET{M,M^\ast}, \langle\,\rangle  \right)$ is valid, we set
   $\Tbf := \left[\vphantom{A^A}[\Tbf\cup \Tbf_M]_\equiv \right]_\simeq$ and then
  continue with Step~\ref{step:start}.
\end{step}

 Since $M(K_4) = \left( M(K_4) \right)^\ast$, we have that $M(K_4)$ is neither a minor of $M$ nor of $M^\ast$
when reaching the next step.

\begin{step}\label{step:chosen}\PRFR{Mar 31st}
  Determine whether $M$ has a minor that is isomorphic to $U_{2,4}$. If this is not the case,  then 
  $\Tbf_M = \left( M, \SET{M,M^\ast}, \emptyset, \emptyset,  \langle\,\rangle  \right)$ is valid, we set
   $\Tbf := \left[\vphantom{A^A}[\Tbf\cup \Tbf_M]_\equiv \right]_\simeq$ and then
  continue with Step~\ref{step:start}.
\end{step}

 Since $U_{2,4} = \left( U_{2,4} \right)^\ast$, we have that $U_{2,4}$ is neither a minor of $M$ nor of $M^\ast$
when reaching the next step.

\begin{step}\PRFR{Mar 31st}
  If $M\in \Mcal$, continue immediately with Step~\ref{step:nonStrict}.
  If $\alpha_M \geq 0$, then 
$\Tbf_M = \left( M, \SET{M,M^\ast}, \emptyset, \emptyset,  \langle\,\rangle  \right)$ is valid, so we may set
   $\Tbf := \left[\vphantom{A^A}[\Tbf\cup \Tbf_M]_\equiv \right]_\simeq$ and
  continue with Step~\ref{step:start}.
\end{step}

\begin{step}\label{step:nonStrict}\PRFR{Mar 31st}
  If $M^\ast\in \Mcal$, continue immediately with Step~\ref{step:baseorderable}.
  If $\alpha_{M^\ast} \geq 0$, then
$\Tbf_{M^\ast} = \left( M^\ast, \SET{M,M^\ast}, \emptyset, \emptyset,  \langle\,\rangle  \right)$ is valid, so we may set
   $\Tbf := \left[\vphantom{A^A}[\Tbf\cup \Tbf_{M^\ast}]_\equiv \right]_\simeq$ and
  continue with Step~\ref{step:start}.
\end{step}

\begin{step}\label{step:baseorderable}\PRFR{Mar 31st}
  Determine whether $M$ is strongly base-orderable. If this is not the case, then 
  $\Tbf_M = \left( M, \emptyset, \emptyset, \SET{M,M^\ast}, \langle\,\rangle  \right)$ is valid, we set
   $\Tbf := \left[\vphantom{A^A}[\Tbf\cup \Tbf_M]_\equiv \right]_\simeq$ and then
  continue with Step~\ref{step:start}.
\end{step}

 The class of strong base-orderable matroids is closed under duality and minors \cite{Ingleton1971}, therefore $M^\ast$ and
all minors of $M$ and $M^\ast$ are strongly base-orderable upon reaching the next step.

\begin{step}\PRFR{Mar 31st}
  Let $M=(E,\Ical)$. Determine whether there is some $X\in \Ical$ with $\left| X \right| = \rk_M(E) - 3$
  and some $Y\subseteq E\BS X$ such that $\alpha_{M\contract\left( E\BS X \right)}(Y) < 0$. If this is the
  case, then the tableau 
  $\Tbf_M = \left( M, \emptyset, \emptyset, \SET{M,M^\ast}, \langle\,\rangle  \right)$ is valid, so we may set
   $\Tbf := \left[\vphantom{A^A}[\Tbf\cup \Tbf_M]_\equiv \right]_\simeq$ and then
  continue with Step~\ref{step:start}.
\end{step}

\begin{step}\PRFR{Mar 31st}
  Let $M^\ast=(E,\Ical^\ast)$. Determine whether there is some $X\in \Ical^\ast$ with \linebreak 
   $\left| X \right| = \rk_{M^\ast}(E) - 3$
  and some $Y\subseteq E\BS X$ such that $\alpha_{M^\ast\contract\left( E\BS X \right)}(Y) < 0$. If this is the
  case, then the tableau 
  $\Tbf_{M^\ast} = \left( M^\ast, \emptyset, \emptyset, \SET{M,M^\ast}, \langle\,\rangle  \right)$ is valid, we may set
   $\Tbf := \left[\vphantom{A^A}[\Tbf\cup \Tbf_{M^\ast}]_\equiv \right]_\simeq$ and then
  continue with Step~\ref{step:start}.
\end{step}

\begin{step}\PRFR{Mar 31st}
  Determine whether $M$ is deflated. If not, then find a deflate $N$ of $M$ with a ground set of minimal cardinality,
  set $\Tbf := \left[\vphantom{A^A}\left[\left( \Tbf\cup\Tbf_N \right)(M \simeq N)\right]_\equiv \right]_\simeq$
  where $$\Tbf_N = \begin{cases}[r]
     \left( N,\SET{N,N^\ast},\emptyset,\emptyset,\langle\,\rangle \right)& \quad \text{if~}\alpha_N \geq 0,\\
      \left( N,\emptyset,\SET{N},\emptyset,\langle\,\rangle \right) & \quad \text{otherwise,}
      \end{cases}$$ 
    and continue with Step~\ref{step:start}.
\end{step}

\begin{step}\PRFR{Mar 31st}\label{lastStep}
  Determine whether $M^\ast$ is deflated. If not, then find a deflate $N$ of $M^\ast$ with a ground set of minimal cardinality,
  set $\Tbf := \left[\vphantom{A^A}\left[\left( \Tbf\cup\Tbf_N \right)(M^\ast \simeq N)\right]_\equiv \right]_\simeq$
  where $$\Tbf_N = \begin{cases}[r]
     \left( N,\SET{N,N^\ast},\emptyset,\emptyset,\langle\,\rangle \right)& \quad \text{if~}\alpha_N \geq 0,\\
      \left( N,\emptyset,\SET{N},\emptyset,\langle\,\rangle \right) & \quad \text{otherwise,}
      \end{cases}$$ 
    and continue with Step~\ref{step:start}.
\end{step}

\needspace{4\baselineskip}
\begin{step}\label{step:exhaustion}\PRFR{Mar 31st}
  Try to find an extension $N$ of $M$ with at most $\rk_G(E)^2\cdot \left| E \right| + \rk_G(E) + \left| E \right|$ elements
  such that $N$ is not isomorphic to any $M'\in \Gcal \cup \Mcal \cup \Xcal$.
  Set $\Tbf := \left[\vphantom{A^A}\left[\Tbf\cup\Tbf_N\right]_\equiv \right]_\simeq$
  where $$\Tbf_N = \begin{cases}[r]
     \left( N,\SET{N,N^\ast},\emptyset,\emptyset,\langle\,\rangle \right)& \quad \text{if~}\alpha_N \geq 0,\\
      \left( N,\emptyset,\SET{N},\emptyset,\langle\,\rangle \right) & \quad \text{otherwise,}
      \end{cases}$$ 
    and continue with Step~\ref{step:start}, or repeat this step and add multiple extensions of $M$.
  If no such extension of $M$ exists, then set $M:= G$ and continue with Step~\ref{step:chosen}.
\end{step}

\PRFR{Mar 31st}
 Clearly, if we continue this process long enough, then Step~\ref{step:exhaustion} ensures 
that the tableau $\Tbf$ will eventually
become decisive for $G$ by exhausting all isomorphism classes of extensions of $G$
with at most $\rk_G(E)^2\cdot \left| E \right| + \rk_G(E) + \left| E \right|$ elements.
If the procedure is carried out in a parallel fashion, not all spawned threads have to carry out Step~\ref{step:exhaustion},
as long as it is guaranteed that the step is carried out again and again eventually by some threads.

% ACKNOWLEDGEMENTS

\bigskip
\footnotesize
\noindent\textit{Acknowledgments.}
This research was partly supported by a scholarship granted by the FernUniversität in Hagen.

% BIBLIOGRAPHY

\bibliographystyle{plain}

\bibliography{references.bib}

\end{document}